\documentclass[12pt]{article}
\usepackage{geometry}
\usepackage{float}
\usepackage{latexsym,bm}
\usepackage{amsmath,amsfonts,amsthm,amssymb,bbding,mathrsfs,mathtools}
\usepackage{fancyhdr}
\usepackage{graphicx}
\usepackage{hyperref}
\usepackage{xcolor}

\newtheorem{thm}{Theorem}[section]

\newtheorem{lemma}{Lemma}[section]

\newtheorem{obs}{Observation}[section]

\newtheorem{claim}{Claim}[section]
\newtheorem{definition}{Definition}[section]

\newcommand{\ex}{\mathrm{ex}}

\title{Planar Tur\'an number of disjoint union of $C_3$ and $C_5$}
\author{Luyi Li
\and Ping Li
\and Guiying Yan
\and Qiang Zhou
}

\date{}

\begin{document}

\maketitle

\begin{abstract}
The planar Tur\'an number of $H$, denoted by $\ex_{\mathcal{P}}(n,H)$, is the maximum number of edges in an $n$-vertex $H$-free planar graph. The planar Tur\'an number of $k\geq 3$ vertex-disjoint union of cycles is the trivial value $3n-6$.
Let $C_{\ell}$ denote the cycle of length $\ell$ and $C_{\ell}\cup C_t$ denote the union of disjoint cycles $C_{\ell}$ and $C_t$.
The planar Tur\'an number $\ex_{\mathcal{P}}(n,H)$ is known if $H=C_{\ell}\cup C_k$, where $\ell,k\in \{3,4\}$. In this paper, we determine the value $\ex_{\mathcal{P}}(n,C_3\cup C_5)=\lfloor\frac{8n-13}{3}\rfloor$ and characterize the extremal graphs when $n$ is sufficiently large.
\end{abstract}

\maketitle

\section{Introduction} 
Tur\'an-type problem is determining the maximum number of edges in a hereditary family of graphs, which is one of the most central topics in extremal graph theory. Let $\ex(n,H)$ denote the maximum number of edges among all $n$-vertex $H$-free graphs. In 1941, Tur\'an~\cite{turan} gave the exact value of $\ex(n,K_{r})$ and characterized the only extremal graph is the balanced complete $(r-1)$-partite graph, where $K_{r}$ is a complete graph with $r$ vertices. Later in 1946, Erd\H{o}s and Stone~\cite{erdos1946} generalized this result asymtotically by proving $\ex(n,H)=(1-\frac{1}{\chi(H)-1}+o(1))\binom{n}{2}$ for an arbitrary graph $H$, where $\chi(H)$ denotes the chromatic number of $H$.

In 2016, Dowden~\cite{dowden2016} initiated the study of Tur\'an-type problems when host graphs are planar graphs. We use $\ex_{\mathcal{P}}(n,H)$ to denote the maximum number of edges among all $n$-vertex $H$-free planar graphs. Dowden studied the planar Tur\'an numbers of $C_4$ and $C_5$, where $C_k$ is a cycle with $k$ vertices. Ghosh, Gy\H{o}ri, Martin, Paulos and Xiao~\cite{ghosh2022c6} gave the exact value for $C_6$. Shi, Walsh and Yu~\cite{shi2025planar}, and Gy\H{o}ri, Li and Zhou~\cite{győri2023c7} gave the exact value for $C_7$, independently. Until now, the planar Tur\'an number of $C_{k}$ is still unknown for $k\geq 8$.

Cranston, Lidick\'{y}, Liu and Shantanam~\cite{daniel2022counterexample} first gave the lower bound for general cycles as $\ex_{\mathcal{P}}(n,C_k) > \frac{3(k-1)}{k}n-\frac{6(k+1)}{k}$ for $k\geq 11$ and sufficiently large $n$. Lan and Song~\cite{lan2022improved} improved the lower bound to $\ex_{\mathcal{P}}(n,C_{k}) > (3-\frac{3-\frac{2}{k-1}}{k-6+\lfloor \frac{k-1}{2}\rfloor})n+c_{k}$ for $n\geq k\geq 11$, where $c_{k}$ is a constant. Recently, Gy\H{o}ri, Varga and Zhu~\cite{gyHori2024new} improved the lower bound by giving a new construction to show that $\ex_{\mathcal{P}}(n,C_{k}) \geq 3n-6-\frac{6\cdot3^{\log_23}n}{k^{\log_23}}$ for $k\geq 7$ and sufficiently large $n$. Shi, Walsh and Yu~\cite{shi2025dense} gave an upper bound and proved that $\ex_{\mathcal{P}}(n,C_{k}) \leq 3n-6-\frac{n}{4k^{\log_23}}$ for all $n,k\geq4$.

We use $tC_k$ to denote the union of $t$ vertex-disjoint $k$-cycles and $t\mathcal{C}$ to denote the union of $t$ vertex-disjoint cycles without length restrictions. Lan, Shi and Song~\cite{lan2019hfree} showed that $\ex_{\mathcal{P}}(n,t\mathcal{C})=3n-6$ for $t\geq 3$ and the extremal graph is a maximal planar graph, which also indicates that $\ex_{\mathcal{P}}(n,tC_k)=3n-6$ for $t,k\geq 3$. Later, Fang, Lin and Shi~\cite{fang2024} proved that $\ex_{\mathcal{P}}(n,2\mathcal{C})=2n-1$. For the planar Tur\'an number of two vertex-disjoint cycles under length restrictions, Lan, Shi and Song~\cite{lan2024planar} proved that $\ex_{\mathcal{P}}(n,2C_3)=\lceil 5n/2 \rceil - 5$ for $n\geq 6$. Li~\cite{li2024} proved that $\ex_{\mathcal{P}}(n,C_3\cup C_4)=\lfloor 5n/2 \rfloor - 4$ for $n\geq 20$, where $C_3\cup C_4$ denotes the union of two vertex-disjoint cycles $C_3$ and $C_4$. Fang, Lin and Shi~\cite{fang2024} proved that for $n\geq 2661$, $\ex_{\mathcal{P}}(n,2C_4)=19n/7 - 6$ if $7|n$ and $\ex_{\mathcal{P}}(n,2C_4)=\lfloor (19n-34)/7 \rfloor$ otherwise. They also settled the spectral planar Tur\'an number of $2C_k$ for all $k\geq 3$. Recently, Li~\cite{li2025dense} proved that $\ex_{\mathcal{P}}(n,2C_k)=\left[3-\Theta(k^{\log^23})^{-1}\right]n$ when $k\geq 5$ and $n\geq k^{\log_{2}3}$.

Let $G\vee H$ denote the graph obtained from $G$ and $H$ by joining all edges between $V(G)$ and $V(H)$. Let $P_k$ be a path with $k$ vertices.
In this paper, we consider the planar Tur\'an number of $C_3\cup C_5$.

\begin{thm}\label{thm}
    For $n\geq 295660$, $\ex_{\mathcal{P}}(n,C_3\cup C_5)=\lfloor\frac{8n-13}{3}\rfloor$ and the unique extremal graph is $K_2\vee (\lfloor\frac{n-2}{3}\rfloor P_3\cup P_{\{\frac{n-2}{3}\}})$, where $\{\frac{n-2}{3} \}=\frac{n-2}{3}-\lfloor\frac{n-2}{3}\rfloor$.
\end{thm}

\section{Preliminaries}
In this section, we first give some notation and terminology.
All graphs considered in this paper are simple and planar. We denote a simple graph by $G=(V(G),E(G))$ where $V(G)$ is the set of vertices and $E(G)$ is the set of edges. We refer to a planar embedding of a planar graph as \emph{plane graph}. Let $v(G)$ and $e(G)$ denote the number of vertices and edges in $G$, respectively. Let $\delta(G)$ and $\Delta(G)$ denote the minimum degree and maximum degree of $G$. Let $f(G)$ be the number of faces and $f_i(G)$ be the number of faces in $G$, called {\em $i$-faces}, whose boundary is a walk of length $i$. If there is no confusion, then we write $e(G),f(G)$ and $f_i(G)$ by $e,f$ and $f_i$, respectively.
For any subset $S\subseteq V(G)$, the subgraph induced by $S$ is denoted by $G[S]$. We denote by $G\backslash S$ the subgraph induced by $V(G)\backslash S$. If $S=\{v\}$, then we simply write $G\backslash v$. Given two graphs $G$ and $H$, $G$ is $H$-\emph{free} if it contains no subgraph isomorphic to $H$.
Let $W_{k}\coloneqq K_1\vee C_{k-1}$ denote the \emph{wheel} graph and $F_{k}\coloneqq K_1\vee P_{k-1}$ denote the \emph{fan} graph.
For more undefined notation, please refer to~\cite{graph-theory}.
Next, we introduce an important definition, that is \emph{triangular block} proposed in~\cite{ghosh2022c6}.

\begin{definition}
    Let $G$ be a plane graph, a triangular block $B$ in $G$ is defined as follows:
    \begin{enumerate}
        
        \item[(1)] Start with an edge $e$, set $E(B)=\{ e\}$;
        
        \item[(2)] Take $e'\in E(B)$ and search for a $3$-face containing $e'$. Add these other edge(s) in this $3$-face to $E(B)$;
        
        \item[(3)] Repeat step (2) till we cannot find a new $3$-face containing any edge in $E(B)$, then $B$ is called a triangular block.
    
    \end{enumerate}
\end{definition}

Given a triangular block $B$, we could regard $B$ as a plane graph, a {\em hole} of $B$ is a face of size at least four in $B$.
Let $G$ be a plane graph and let $\mathcal{B}$ be the family of all triangular blocks of $G$. Then it is easy to get that $$e(G)=\sum_{B\in \mathcal{B}}e(B) \text{ and } f_3(G)=\sum_{B\in \mathcal{B}}f_3(B).$$

We could classify triangular blocks by the number of vertices. Figure~\ref{triangular_blocks} shows all triangular blocks with at most $5$ vertices, where the subscript represents the number of vertices, and the superscript is used to distinguish triangular blocks that are not isomorphic when there are the same number of vertices:

\begin{figure}[H]
    \centering
    \includegraphics[width=0.6\linewidth]{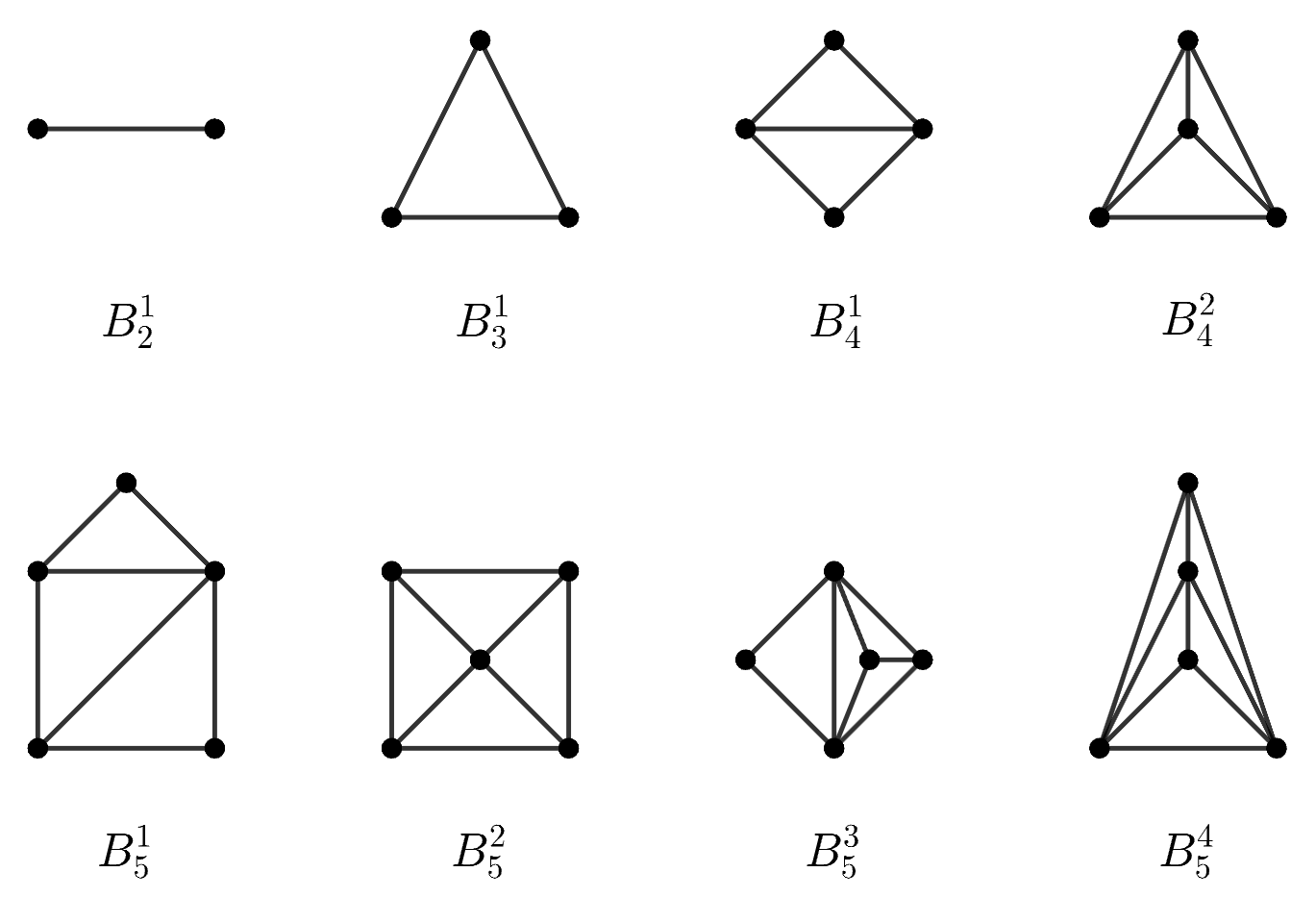}
    \caption{Triangular blocks with at most $5$ vertices.}
    \label{triangular_blocks}
\end{figure}

\section{Lemmas}
In this section, we prove some vital lemmas.
\begin{lemma}\label{lemma1}
    Let $G$ be a plane graph with $e(G)\geq \left\lfloor\frac{8n-13}{3}\right\rfloor$. If $G$ is $C_3\cup C_5$-free and $n\geq 295660$, then one of the following statements holds:
    \begin{enumerate}
        
        \item[(1)] $G$ has at least $3$ triangular blocks containing $C_5$ and intersecting exact two vertices;
        
        \item[(2)] $G$ has a triangular block with at least $521$ vertices;
        
        \item[(3)] $G$ has at least $4$ triangular blocks containing $C_5$ and intersecting exact one vertex.
    
    \end{enumerate}
\end{lemma}

\begin{proof}
    Assume that statements (2) and (3) do not hold. Next we prove that statement (1) holds. Without loss of generality, suppose that $G$ is the planar embedding.
    Since $e(G)>\ex_{\mathcal{P}}(n,C_5)=\frac{12n-33}{5}$ when $n\geq 11$, $G$ contains a copy $D$ of $C_5$. Now we have the following claims:
    
    \begin{claim}\label{claim1}
        There are at most $10$ triangular blocks isomorphic to $B_4^2$ in $G$.
    \end{claim}
    \begin{proof}
        Note that each $B_4^2$ intersects at least two vertices of $D$. Otherwise, we could easily find $C_3\cup C_5$. Since any two blocks of $B_4^2$ intersect at most one vertex, it follows that there are at most $10$ triangular blocks isomorphic to $B_4^2$ in $G$.
    \end{proof}
    
    \begin{claim}\label{claim2}
        There are at least $f(n)$ triangular blocks containing $C_5$, where $f(n)=\frac{4n+15097}{15555}$.
    \end{claim}
    \begin{proof}
        Assume that there are only $\alpha=f(n)-1$ triangular blocks containing $C_5$. By Claim~\ref{claim1}, there are at most $10$ triangular blocks isomorphic to $B_4^2$ in $G$. By the assumption that statement (2) does not hold, the largest triangular block has at most $520$ vertices. The remaining triangular blocks that do not contain $C_5$ and are not isomorphic to $B_4^2$, denoted as $B_1,\ldots,B_t$, are isomorphic to $B_2^1$, $B_3^1$, and $B_4^1$. Then
        $$f_3(G) \leq \sum_{i\in[t]}f_3(B_i)+3\times10+\alpha\times (2\times 520-3) \leq \frac{2}{5}\sum_{i\in [t]}e(B_i)+1037\alpha+30 \leq \frac{2}{5}e(G)+1037\alpha+30,$$
        By Euler's formula, we have
        $$2e = \sum_{i}if_i \geq 3f_3+4(f-f_3) \geq 4(e+2-n)-\frac{2e}{5}-1037\alpha-30,$$
        which implies $e \leq \frac{5}{8}(4n+1037\alpha+22)=\frac{8n-16}{3}$, thus $e<\left\lfloor\frac{8n-13}{3}\right\rfloor$, a contradiction.   
    \end{proof}
    Since $G$ is $C_3\cup C_5$-free, each triangular block containing $C_5$ intersects with $D$.
    Recall that statement (3) does not hold. It follows that there are at most $3$ triangular blocks containing $C_5$ and intersecting exact one fixed vertex with $D$. Since $G$ is a plane graph, for any subset $A\subseteq V(D)$ with $|A|\geq 3$, there are at most two triangular blocks $B_1,B_2$ such that $B_i$ contains a $5$-cycle and $B_i\cap V(D)=A$ for $i=1,2$. Thus, there are at least 
    $$\frac{f(n)-3\times 5-2\big[\binom{5}{3}+\binom{5}{4}+\binom{5}{5}\big]}{\binom{5}{2}}=\frac{2n-357994}{77775}\geq 3$$ 
    triangular blocks that contain $C_5$ and intersect exactly two vertices. The result follows.
\end{proof}

A triangular block $B$ with at least $6$ vertices is \emph{good} if for any two nonadjacent vertices $u$ and $v$ in a hole of $B$, there exists a $C_5$ in $B$ such that $C_5$ contains at most one vertex of $\{u,v\}$. Otherwise, $B$ is \emph{bad}.





\begin{obs}\label{obs}
    Let $B$ be a good triangular block. Assume that $B^*$ is a triangular block obtained from $B$ by adding a new vertex in a hole of $B$. Then $B^*$ is also good.
\end{obs}

\begin{proof}
    If $u,v$ are not adjacent and in a hole of $B^*$, then $u$ or $v$ is the new added vertex since $B$ is good. Let $u$ be the new added vertex. Then $v$ is contained in some $5$-cycle in $B$ since $|B|\geq 6$, a contradiction.
\end{proof}

\begin{figure}[H]
    \centering
    \includegraphics[width=0.5\linewidth]{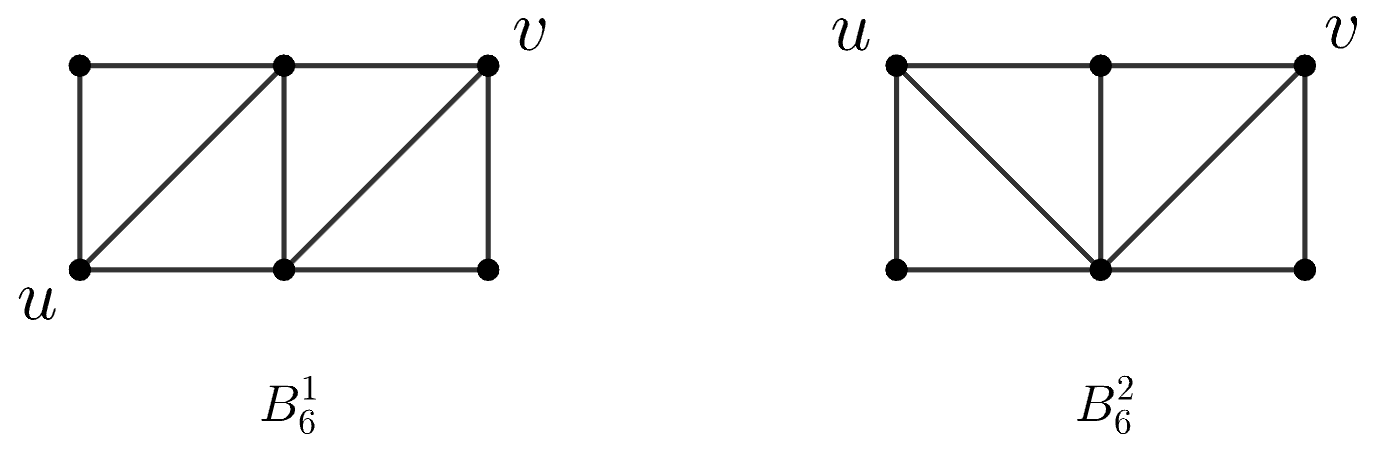}
    \caption{Two special triangular blocks in Lemma~\ref{lemma2}}
    \label{2_special_blocks}
\end{figure}

\begin{lemma}\label{lemma2}
    For a triangular block $B$ with at least $6$ vertices, $B$ is good or $B\in\{B_6^1,B_6^2\}$.
\end{lemma}

\begin{proof}
    We could obtain all triangular blocks with $6$ vertices by adding a new vertex to each triangular block with $5$ vertices.

    There are $9$ triangular blocks formed by adding a new vertex to $B_{5}^{1}$ as shown in Figure~\ref{B51}.

    \begin{figure}[H]
        \centering
        \includegraphics[width=0.9\linewidth]{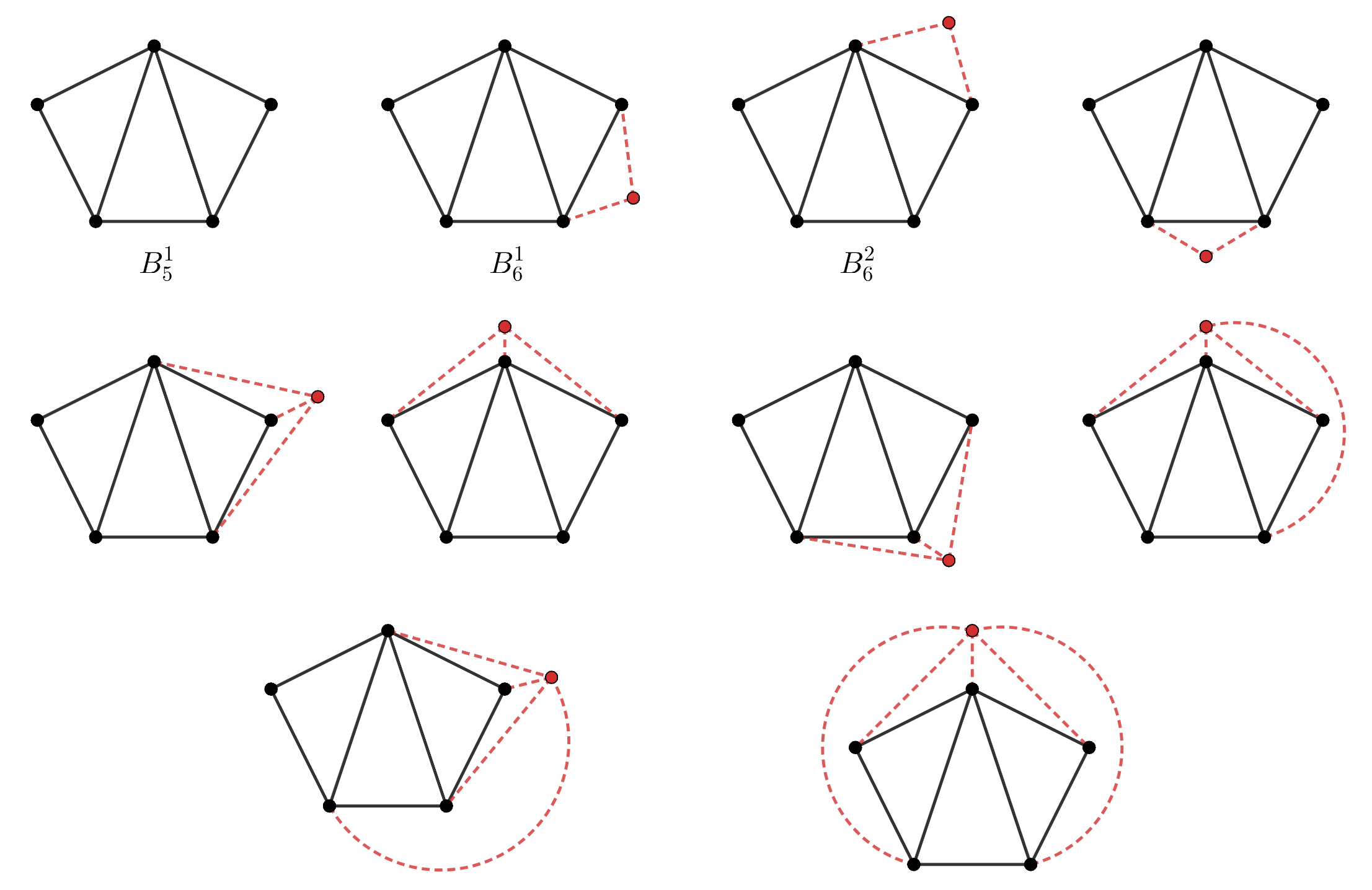}
        \caption{Triangular blocks with $6$ vertices coming from $B_{5}^{1}$}
        \label{B51}
    \end{figure}

    There are $3$ triangular blocks formed by adding a new vertex to $B_{5}^{2}$ as shown in Figure~\ref{B52}.
    
    \begin{figure}[H]
        \centering
        \includegraphics[width=0.9\linewidth]{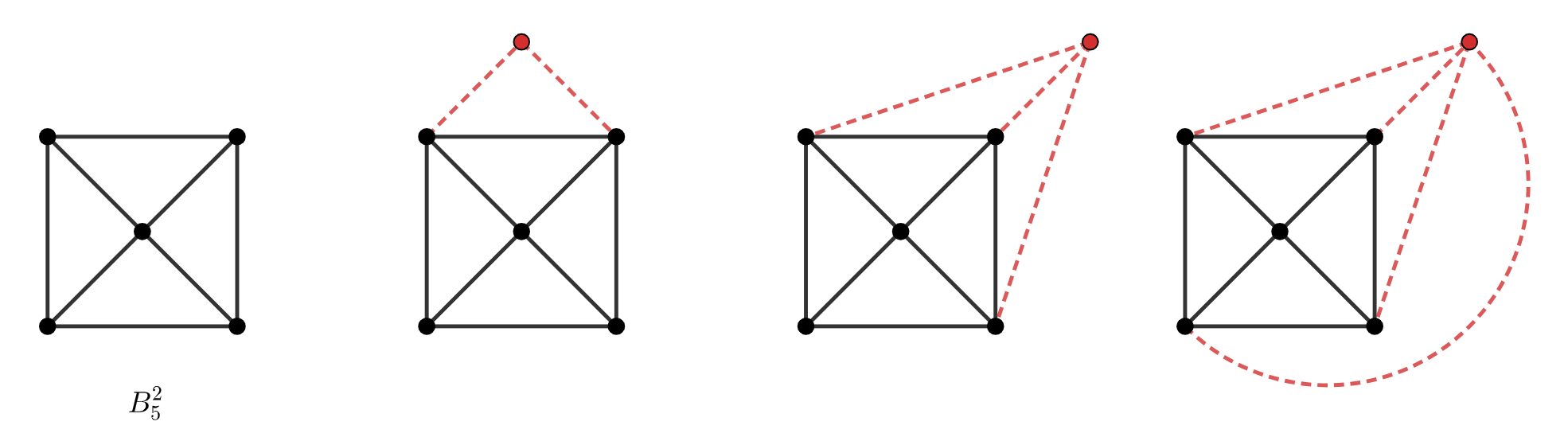}
        \caption{Triangular blocks with $6$ vertices coming from $B_{5}^{2}$}
        \label{B52}
    \end{figure}

    There are $6$ triangular blocks formed by adding a new vertex to $B_{5}^{3}$ as shown in Figure~\ref{B53}.

    \begin{figure}[H]
        \centering
        \includegraphics[width=0.9\linewidth]{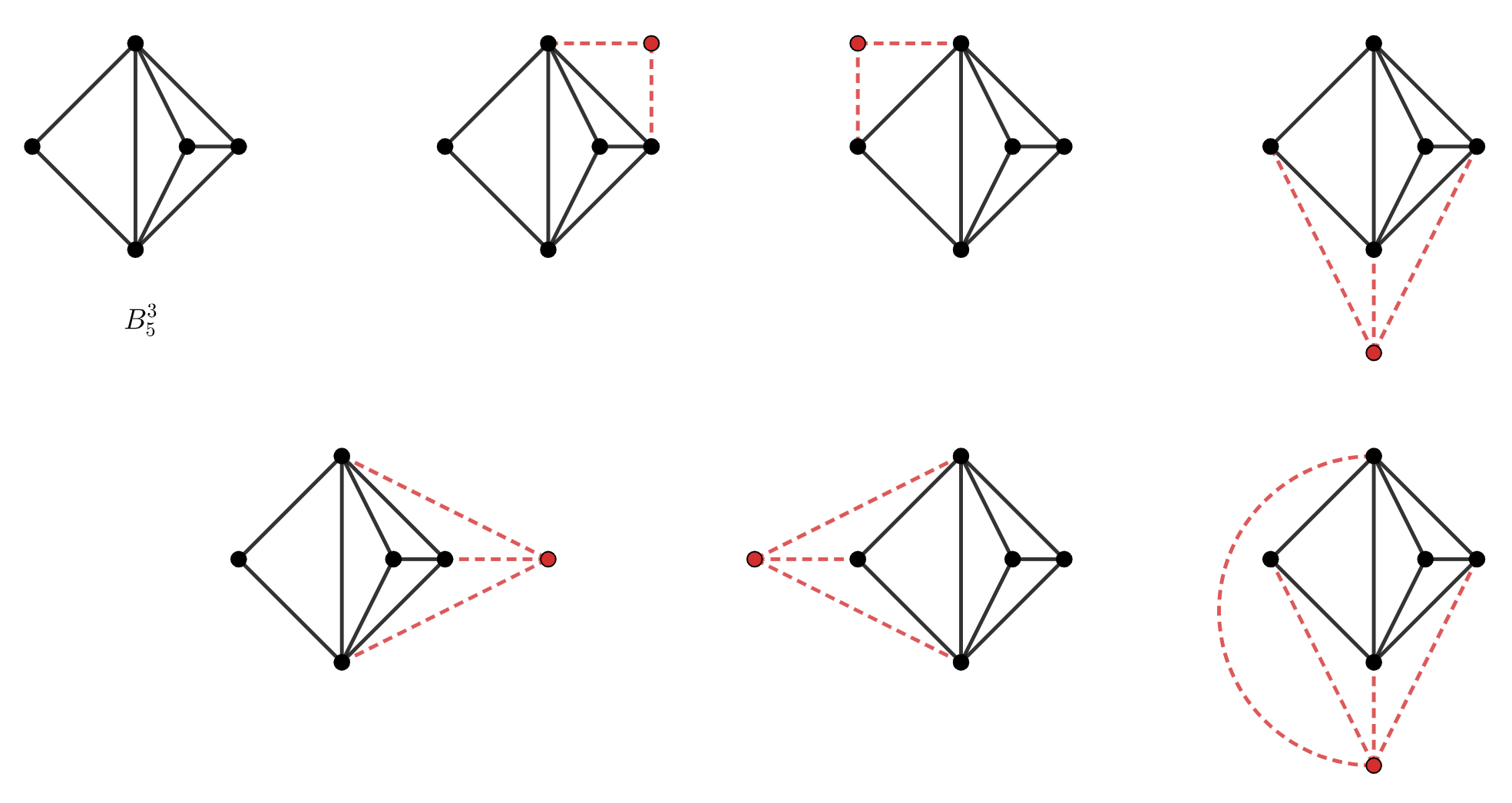}
        \caption{Triangular blocks with $6$ vertices coming from $B_{5}^{3}$}
        \label{B53}
    \end{figure}

    There are $3$ triangular blocks formed by adding a new vertex to $B_{5}^{4}$ as shown in Figure~\ref{B54}.

    \begin{figure}[H]
        \centering
        \includegraphics[width=0.8\linewidth]{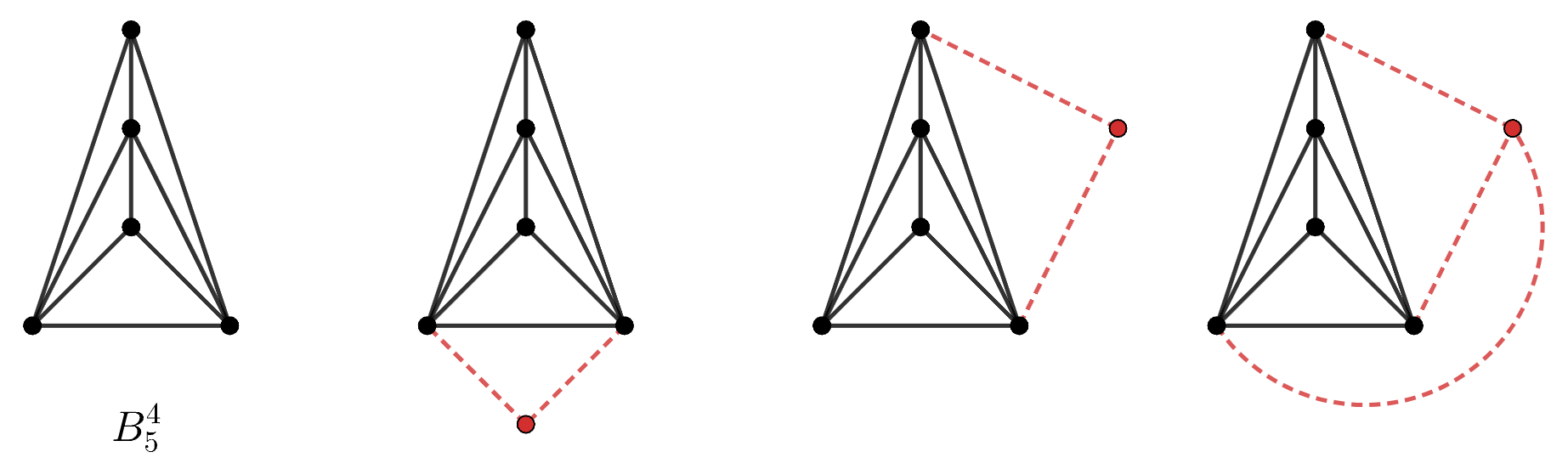}
        \caption{Triangular blocks with $6$ vertices coming from $B_{5}^{4}$}
        \label{B54}
    \end{figure}

    For the second and third graphs ($B_6^1$ and $B_6^2$) in Figure~\ref{B54}, it is not difficult to find that only these two triangular blocks are bad.
    Note that a triangular blocks with $7$ vertices obtained from $B_6^i$ by adding a new vertex in a hole of $B_6^i$ is good for $i=1,2$.
    From Observation~\ref{obs}, only $B_{6}^{1}$ and $B_{6}^{2}$ shown in Figure~\ref{2_special_blocks} are bad.
\end{proof}

\begin{lemma}\label{lemma3}
    Suppose that $G$ satisfies statement (1) of Lemma~\ref{lemma1}. Let $u,v$ be the two vertices. Let $B_1,\ldots,B_t$ be the triangular blocks containing $\{u,v\}$ after removing the edge $e_{uv}$(if exists). Then $B_i\in\{B_4^1,B_5^1,B_5^2,B_5^3,B_6^1,B_6^2\}$ for $i=1,2,...,t$.
\end{lemma}

\begin{proof}
    Since there is at most one triangular block will break after removing the edge $e_{uv}$(if exists), then $t\geq 2$. Without loss of generality, if $B_1$ is not any given triangular block. Then by Lemma~\ref{lemma2}, we could find a $C_5$ in $B_1$ that only intersects at most one vertex of $\{ u, v \}$. Then we could find $C_3\cup C_5$ in $B_1$ and $B_2$.
\end{proof}

\begin{lemma}\label{lemma4}
    If $G$ is $C_3\cup C_5$-free and contains one triangular block $B$ with $|V(B)|\geq 521$, then $B$ is a wheel or a fan.
\end{lemma}

\begin{proof}
    Suppose to the contrary that $B$ is not a wheel or a fan. Since $|V(B)|\geq 5$ and $B$ is a triangular block, $B$ contains cycles of length five. If $\Delta(B)\leq 105$, then there exists a vertex $u$ such that the distance between $u$ and some $5$-cycle $C$ in $B$ is $2$, since $|V(B)|\geq 521$. By the definition of triangular block, $u$ is contained in some $3$-face. However, this $3$-face is not intersecting with $C$, which contradicts to the hypothesis that $G$ is $C_3\cup C_5$-free. Thus we may assume that $\Delta(B)\geq 106$. Let $v$ be a vertex in $B$ with $d_B(v)=\Delta(B)$.
    
    If $B$ contains a fan $\{u\}\vee P_{16}$, then all $3$-faces must intersect at least one vertex of $\{ u, u_1, u_2, u_3, u_4 \}$, $\{ u, u_5, u_6, u_7, u_8 \}$, $\{ u, u_9, u_{10}, u_{11}, u_{12} \}$ and $\{ u, u_{13}, u_{14}, u_{15}, u_{16} \}$, where $P_{16}=u_1u_2\ldots u_{16}$. Thus, all $3$-faces intersect $u$, which means that $B$ is a wheel or a fan with the centre $u$. So we may assume that $B$ is $(K_1\vee P_{16})$-free.

    Since $d_B(v)\geq 106$, all $3$-faces incident to $v$ can be partitioned into $t$ fans $\{F^i:i\in[t]\}$ such that the centre of each $F^i$ is $v$. See Figure~\ref{partition} as an example.

    \begin{figure}[H]
        \centering
        \includegraphics[width=0.35\linewidth]{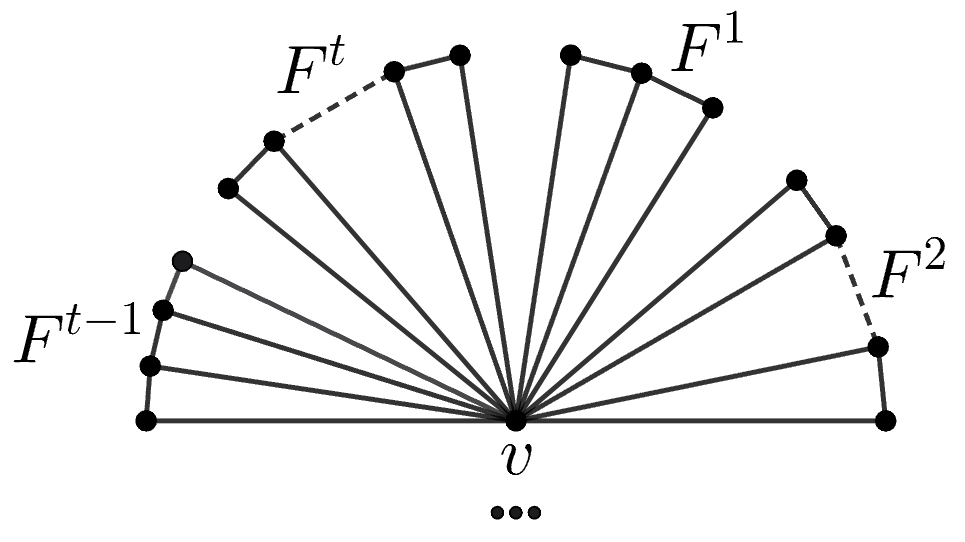}
        \caption{The partition of $3$-faces incident to $v$}
        \label{partition}
    \end{figure}

    Recall that $B$ is $(K_1\vee P_{16})$-free, we have $t\geq 8$.
    All $5$-cycles in $B$ must contain $v$, otherwise there is a $5$-cycle $C^*$ containing no $v$. Then $C^*$ intersects at least one vertex of $F^i$ for each $i\in [t]$, a contradiction. Thus we may fix a $5$-cycle $C$ containing $v$. Then there are at least $(t-4)$ fans, say $F^1,F^2,...,F^{t-4}$, containing none vertex of $C\backslash v$. Since $B$ is a triangular block, each $F^i$ must be extended by $3$-faces. These extended $3$-faces could not intersect with $v$ since all $3$-faces incident to $v$ are determined. Since $B$ is $C_3\cup C_5$-free, these extended $3$-faces must intersect with $C$. See Figure~\ref{partition_extend} as an example.
    It follows from the planarity of $B$  that there are at most three $F^i$'s such that each of them has at least $2$ neighbors in $C\backslash v$. Now there still exists at least one fan $F^j$ such that $F^j$ has only one neighbor in $C\backslash v$. Let $u_1u_2w$ be the extended $3$-face of $F^j$ and $w\in C\backslash v$, as shown in Figure~\ref{partition_extend}.

    To ensure that $B$ is a triangular block, $u_1w$ or $u_2w$ is belonging to two $3$-faces. Without loss of generality, let $u_1wx$ be a $3$-face and $x\notin C$ by the choice of $F^j$. But the known edges could not form a triangular block, one of $\{u_1x, wx, u_2w\}$ should belong to two $3$-faces. No matter who is belonging to two $3$-faces, let $y$ be the new vertex in the new $3$-face with $y\neq v$. Then $u_1$, $u_2$, $w$, $x$ and $y$ form a $5$-cycle containing no $v$, a contradiction.
\end{proof}

\begin{figure}[H]
    \centering
    \includegraphics[width=0.3\linewidth]{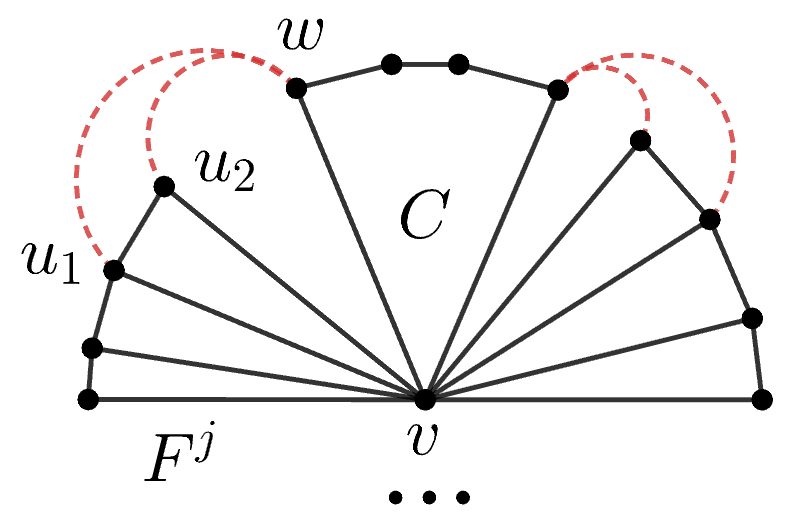}
    \caption{Each part should extend $3$-faces intersecting with $C$}
    \label{partition_extend}
\end{figure}

\begin{lemma}\label{lemma5}
    If one of the statements (2) and (3) in Lemma~\ref{lemma1} holds, then all $3$-cycles in $G$ intersect one vertex.
\end{lemma}

\begin{proof}
    If the statement (2) of Lemma~\ref{lemma1} holds, then $G$ has a triangular block $B$ with $t\geq 521$ vertices. By Lemma~\ref{lemma4}, $B$ is a wheel $\{u\}\vee P_{t-1}$ or a fan $\{u\}\vee C_{t-1}$, where $P_{t-1}=u_1u_2\cdots u_{t-1}$ and $C_{t-1}=u_1u_2\cdots u_{t-1}u_1$. Note that each $3$-cycle must intersect one vertex of $\{v, u_1, u_2, u_3, u_4 \}$, $\{v, u_5, u_6, u_7, u_8 \}$, $\{v, u_9, u_{10}, u_{11}, u_{12} \}$ and $\{v, u_{13}, u_{14}, u_{15}, u_{16} \}$, respectively. So all $3$-cycles in $G$ intersect $v$.

    If the statement (3) of Lemma~\ref{lemma1} holds, then there are four triangular blocks containing $C_5$ and intersecting exact one vertex $v$. Then all $3$-cycles in $G$ intersect $v$, otherwise we could find a $C_3\cup C_5$ in $G$, a contradiction.
\end{proof}

\section{Proof of Theorem~\ref{thm}}

Let $G$ be an $n$-vertex $C_3\cup C_5$-free plane graph with $e(G)>\frac{8n-13}{3}$ and $n\geq 295660$. If Lemma~\ref{lemma1} (2) or (3) holds, then by Lemma~\ref{lemma5}, all $3$-cycles in $G$ intersect one vertex $v$. Thus, all $3$-faces in $G$ incidenct with $v$.
Note that $f_3\leq n-1$.
By Euler's formula, we have
$$2e = \sum_{i\geq 3}if_i \geq 3f_3+4(f-f_3)=4f-f_3 \geq 4(e+2-n)-(n-1),$$
which implies $e \leq \frac{5n-9}{2} < \frac{8n-13}{3}$, a contradiction.

From Lemma~\ref{lemma1}, $G$ has at least $3$ triangular blocks containing $C_5$ and intersecting exact two vertices $u$ and $v$. 
Since $G$ is $C_3\cup C_5$-free, $u$ and $v$ belong to every triangular block containing $C_5$. 
Let $G'=G-e_{uv}$ if $e_{uv}$ exists and $G'=G$ otherwise.  Assume that $B_1,\ldots,B_t$ are all triangular blocks of $G'$.
By Lemma~\ref{lemma3}, $f_3(B_i) \leq e(B_i)/2$ with equality holds when $B_i\in \{B_5^2,B_5^3\}$ for $i\in [t]$.
Thus,
$$e(G') = \sum_{i\in [t]}e(B_{i}) \geq 2\sum_{i\in [t]}f_3(B_{i}) = 2f_3(G').$$
Then by Euler's formula, we have
$$2e(G') = \sum_{i\geq 3}if_i(G') \geq 3f_3(G')+4(f(G')-f_3(G'))=4f(G')-f_3(G') \geq 4(e(G')+2-n)-\frac{e(G')}{2},$$
which implies $e(G') \leq \frac{8n-16}{3}$. Then $e(G) \leq \frac{8n-13}{3}$, a contradiction. 
The proof is finally complete. $\hfill\qedsymbol$

\begin{figure}[H]
    \centering
    \includegraphics[width=0.8\linewidth]{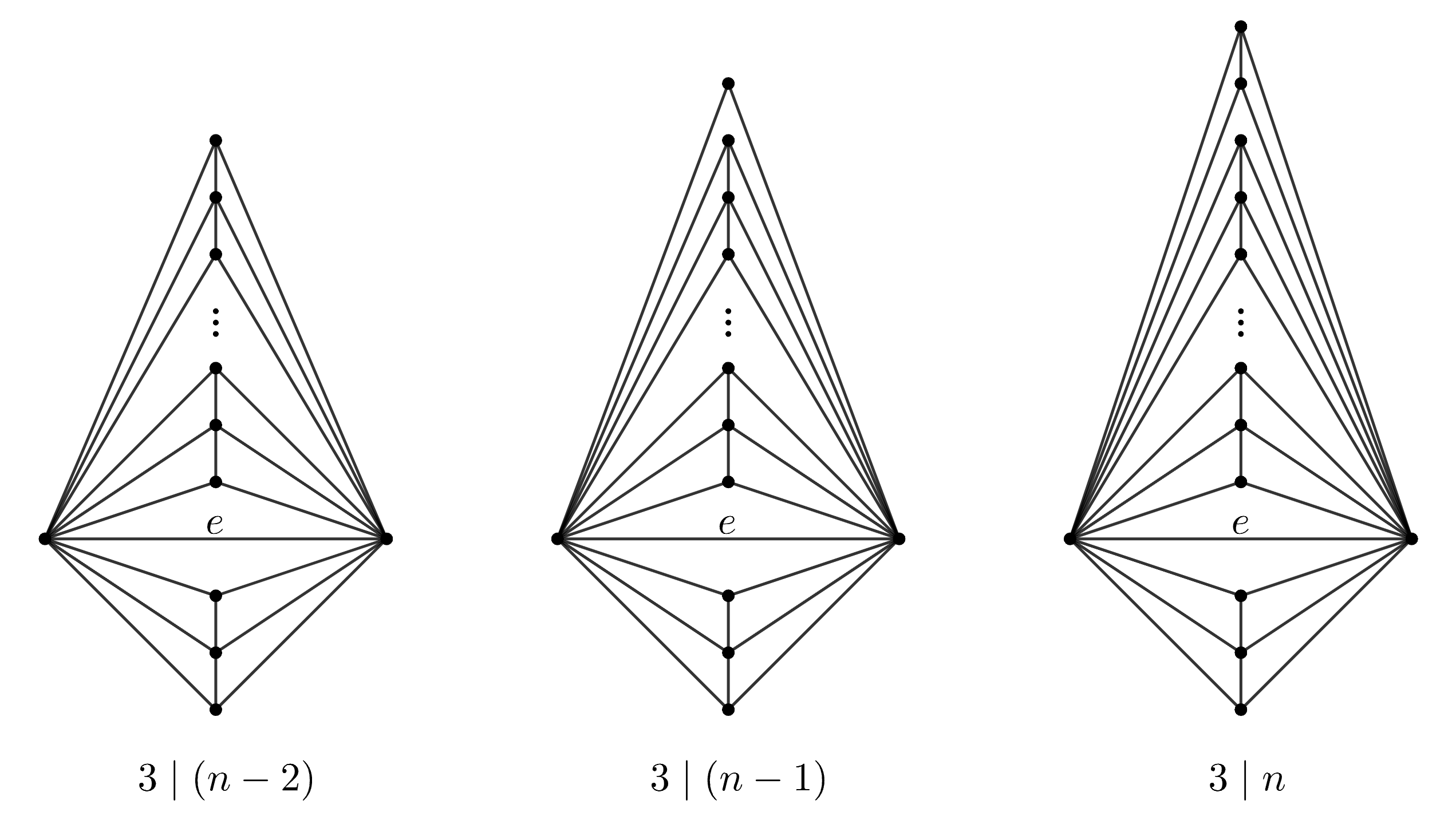}
    \caption{Extremal graphs of Theorem~\ref{thm}}
    \label{extremal_graphs}
\end{figure}

Now let $G$ be an $n$-vertex $C_3\cup C_5$-free plane graph with $e(G)=\frac{8n-13}{3}$ and $3 \mid (n-2)$. Through the discussion above, there must exist an edge $e=\{u,v\}$ and all triangular blocks in $G-e$ are $B_5^2$ or $B_5^3$. If there exists a $B_5^3$, then there is a triangle inside this $B_5^3$ that does not intersect with the $5$-cycle containing $u$ and $v$, leading to a contradiction. So, the unique extremal graph is shown in the first one of Figure~\ref{extremal_graphs} when $n$ is large enough. For the case where $3 \mid (n-1)$ or $3 \mid n$, by a similar discussion we could get that the remaining graphs in Figure~\ref{extremal_graphs} are the only extremal graphs for each case.

\section{Acknowledgments}

This work was supported by the National Key R\&D Program of China (No. 2023YFA1009602) and the National Natural Science Foundation of China (Grant No. 12231018).

\bibliographystyle{abbrv}
\bibliography{main}
~~\\
~~\\
~~\\

(Luyi Li) Academy of Mathematics and Systems Science, Chinese Academy of Sciences, Beijing 100190, China. 
{\bf Email address:} \href{liluyiplus@gmail.com}{liluyiplus@gmail.com}\\

(Ping Li) School of Mathematics and Statistics, Shaanxi Normal University, Xi'an, Shaanxi 710062, China. 
{\bf Email address:} \href{lp-math@snnu.edu.cn}{lp-math@snnu.edu.cn}\\

(Guiying Yan) Academy of Mathematics and Systems Science, Chinese Academy of Sciences, and University of Chinese Academy of Sciences, Beijing 100190, China.
{\bf Email address:} \href{yangy@amss.ac.cn}{yangy@amt.ac.cn}\\

(Qiang Zhou) Academy of Mathematics and Systems Science, Chinese Academy of Sciences, and University of Chinese Academy of Sciences, Beijing 100190, China. {\bf Email address:} \href{zhouqiang2021@amss.ac.cn}{zhouqiang2021@amss.ac.cn}\\

\end{document}